\documentclass[11pt,reqno]{amsart}
\usepackage{amsthm}
\usepackage{amsmath}
\usepackage{amsfonts}
\usepackage{amssymb}
\usepackage{verbatim}
\usepackage[all,cmtip]{xy}
\usepackage{color}
\usepackage{comment}
\usepackage{enumerate}
\newtheorem{thm}{Theorem}[section]

\newtheorem{remark}[thm]{Remark}
\newtheorem{lemma}[thm]{Lemma}
\newtheorem{prop}[thm]{Proposition}

\newtheorem{ex}[thm]{Example}

\newtheorem{defn}[thm]{Definition}

 \numberwithin{equation}{section}
\newcommand{\bb}[1]{\mathbb{#1}}
\newcommand{\cl}[1]{\mathcal{#1}}

\newcounter{egcounter}
\begin{document}
\title[nearly invariant brangesian subspaces]{nearly invariant brangesian subspaces}

\author{Arshad Khan}
\address{Department Of Mathematics\\
        Shiv Nadar Institution Of Eminence\\
        School of Natural Sciences\\
         Gautam Budh Nagar - 203207\\
         Uttar Pradesh, India}
\email{ak954@snu.edu.in}

\author{Sneh Lata}
\address{Department Of Mathematics\\
       Shiv Nadar Institution Of Eminence\\
         School of Natural Sciences,\\
         Gautam Budh Nagar - 203207\\
         Uttar Pradesh, India}
\email{sneh.lata@snu.edu.in}

\author{ Dinesh Singh}
\address{Centre For Lateral Innovation, Creativity and Knowledge\\
      SGT University\\
   Gurugram 122505\\
        Haryana, India}
\email{dineshsingh1@gmail.com}


%\date{October 31, 2023}

\begin{abstract}
This article describes Hilbert spaces contractively contained in certain reproducing kernel Hilbert spaces of analytic functions on the open unit disc which are nearly invariant under division by an inner function. We extend Hitt's theorem on nearly invariant subspaces of the backward shift operator on $H^2(\bb D)$ as well as its many generalizations to the setting of de Branges spaces.
\end{abstract}
\maketitle

\section{introduction}
In this paper we study nearly invariant subspaces from a Brangesian point of view. A subspace $\cl M$ of the Hardy space 
$H^2(\bb D)$ is called nearly invariant under the backward shift operator $S^*$ on $H^2(\bb D)$ if $S^*(f)$ belongs to 
$\cl M$ whenever $f$ vanishes at zero. These subspaces first arose in the work of Hitt \cite{Hitt} while characterizing 
the shift invariant subspaces of the Hardy space of an annulus. The kernels of Toeplitz operators are particular examples of nearly $S^*$-invariant 
subspaces, and this special case of Hitt's theorem was independently established by Hayashi \cite{Hya} by developing ideas similar to those used by Hitt. 
Hitt called these subspaces ``weakly invariant" rather than ``nearly invariant". Sarason, \cite{Sar}, coined the term ``nearly invariant subspaces" and-more importantly-gave a new proof of Hitt’s theorem by utilizing ideas based on de Branges-Rovnyak spaces, \cite{deBR}. See also \cite{Sar1}. In doing so, Sarason engendered new ideas that gave rise to some very interesting papers such as \cite{ABBH, AFR, Che, Yak}. Since the time \cite{Hya}, \cite{Hitt} and particularly \cite{Sar} appeared nearly invariant subspaces have established themselves as an important area of research and they can be deemed to be a proper generalization of the concept of invariant subspaces. In addition, they connect with many diverse areas including with mathematical physics. See D. Vukoti$\acute{\rm c}$ (2011). [Review of the book The Hardy spaces of a slit domain, by A. Aleman, N. Feldman, and W. Ross]. MR2548414 (2011m:30095).

\begin{thm}{\bf (Hitt's theorem)}\label{hitt} Let $\cl M$ be a non-trivial nearly invariant subspace of $H^2(\bb D)$ under $S^*,$ and let $g$ be a function in 
$\cl M$ of unit norm that is orthogonal to $\cl M\cap zH^2(\bb D)$ and positive at the origin. Then there exists a $S^*$-invariant subspace 
$\cl N$ of $H^2(\bb D)$ such that $\cl M=g\cl N$ and $||gf||=||f||$ for all $f\in \cl N.$ 
\end{thm}

In 2010, Chalender, Chevrot, and Partington \cite{CCP} generalized Hitt's result to the backward shift operator on a vector-valued Hardy 
space. A few years later, the first and third author from \cite{CCP}, in collaboration with Gallardo-Gutierrez, introduced and described in \cite{CGP} 
nearly invariant subspaces {\em with finite defect} for the backward shift operator on $H^2(\bb D).$ This work was further extended to the 
vector-valued case by Chattopadhyay, Das, and Pradhan in \cite{CDP}.
 
Simultaneously, Erard \cite{Era} in 2004 introduced the notion of nearly invariant subspaces in the vastly general situation of multiplication operators that are bounded 
below on reproducing kernel Hilbert spaces. He first deduced a factorization theorem in the general setting and later used 
it to describe nearly invariant subspaces of the backward shift on general reproducing kernel Hilbert spaces of analytic functions on the open 
unit disc $\bb D$ on which the operator of multiplication with $z$ is well-defined and bounded below. As a particular case, his result also described 
nearly invariant subspaces of the backward shift on $H^2(\bb D).$ His description turns out to be the same as Hitt's. However, since he was working in a much more general setting, his method could not capture two crucial pieces of information about the representation, namely, the norm equality (as it appears in Hitt's theorem) and the closedness of the backward shift invariant subspaces that appear in the representation. In 2021, Liang and Partington used Erard's factorization theorem (\cite[Theorem 3.2]{Era}) to describe nearly invariant subspaces of Dirichlet-type spaces $\cl D_\alpha, \ -1\le\alpha \le 1$ with respect to the operator of multiplication with a finite Blaschke factor. This work of 
Liang and Partington has been extended to the finite defect setting in 2022 by Chattopadhyay and Das in \cite{CD}. 

 Let $\cl H_1$ and $\cl H_2$ be two Hilbert spaces with norms $||\cdot||_1$ and $||\cdot||_2,$ respectively. We say $\cl H_1$ is 
contractively contained in $\cl H_2$ if $\cl H_1$ is a vector subspace (not necessarily closed) of $\cl H_2$ and the inclusion map is a 
contraction, that is, $||h||_2\le ||h||_1$ for all $h\in \cl H_1. $ In this paper we shall investigate the above-mentioned avenues of research associated with nearly invariant subspaces for contractively contained Hilbert spaces. 

The organization of the paper is as follows. Section \ref{def} contains definitions and terminologies that will be used throughout the paper. In Section \ref{Hardy}, we describe Hilbert spaces that are contractively contained in the Hardy space 
$H^2(\bb D, \bb C^n)$ and which are nearly invariant under the backward shift operator on $H^2(\bb D,  \bb C^n ).$ Our result (Theorem \ref{MTH}) is, in a sense, the best possible generalization-in the setting of de Branges spaces-of Hitt’s theorem (Theorem \ref{hitt}) and also its vector-valued generalization by Chalendar et al. (Theorem \ref{CCP}). This is so, since the represenations obtained in both these theorems can be easily derived from our version once we assume that our general de Branges space is the special case of the scalar valued Hardy space of Hitt or the $n$-dimensional valued Hardy space of Chalendar et al. At the same time, we show through specific counterexamples that our characterization per se in the general setting of the contractively contained de Branges space cannot be improved. In other words we show that in the conclusion of Theorem \ref{MTH}, our inequality between the de Branges space and the Hardy space cannot be improved to an equality (Example \ref{EH1}) nor can we conclude in the general case that the backward shift invariant subspace in our description is closed, see (Example \ref{EH2}). Afterwards, in Section 
\ref{rkhs} (Theorem \ref{MTRKHS}), we extend a work of Erard from \cite[Theorem 5.1]{Era} (stated here as Theorem \ref{Era}) to the case of contractively contained Hilbert spaces. Our Theorem \ref{MTRKHS} is in fact also an extension of Liang and Partington's result (\cite[Theorem 3.4]{LP}) that used Erad's result to describe subspaces of the Dirichlet-type spaces $\cl D_{\alpha} \ (0\le \alpha \le 1)$ which are nearly invariant under ``division by a finite Blaschke factor". Lastly, in Section \ref{finite}, we extend our study from Section \ref{rkhs} to the finite defect setting. Theorem \ref{RKHS-defect} is an extension of Chattopadhay and Das' result \cite[Theorem 3.9]{CD} to multiplication with inner function on general reproducing kernel Hilbert spaces which in turn is a generalization of Liang and Partington's above-mentioned work to the finite defect setting.

\section{Terminologies and Definitions}\label{def}
Let $\bb D$ be the open unit disc in the complex plane $\bb C.$ For a given Hilbert space $\cl K,$ let $H^2(\bb D, \cl K)$ denote the familiar Hardy space of $\cl K$-valued analytic functions on $\bb D.$ Recall that 
$$
H^2(\bb D, \cl K) = \left \{\sum_{m=0}^\infty A_m z^m :  A_m\in \cl K,  \ \sum_{m=0}^\infty ||A_m||^2_{\cl K}<\infty\right\}
$$
and it is a Hilbert space with respect to the norm 
$||f||^2_{2,\cl K} = \sum_{m=0}^\infty ||A_m||^2_{\cl K},$ where $f(z)=\sum_{m=0}^\infty A_mz^m$ belongs to $H^2(\bb D, \cl K).$ The Hardy space $H^2(\bb D,\bb C)$ 
is denoted simply as $H^2(\bb D)$. Note that if $\{x_i:i\in I\}$ is an orthonormal basis of $\cl K,$ then $H^2(\bb D,\cl K)$ can be identified (under an isometric isomorphism) with $\ell^2$-direct sum of $I$ copies of $H^2(\bb D).$ Thus, each 
$f\in H^2(\bb D, \cl K)$ can be identified with an $I$-tuple $(f_i)_{i\in I}$, where each $f_i\in H^2(\bb D)$ and $||f||_{2,\cl K}^2=\sum_{i\in i} ||f_i||_{2,\bb C}^2.$ Henceforth, for notational convenience, we shall not mention $\cl K$ in the norm $||\cdot||_{2,\cl K}$; instead, we shall write it as $||\cdot||_2$.  

The forward shift or simply the shift operator $S$ on 
$H^2(\bb D, \cl K)$ is defined as $Sf(z)=zf(z)$ and it's adjoint $S^*$, known as the backward shift operator, is given by 
$$
S^*f(z)=\frac{f(z)-f(0)}{z}
$$
for $f\in H^2(\bb D, \cl K).$ 

In the introduction, we have used the term subspace to refer to a closed subspace, and we shall keep following the same terminology throughout the paper. 
But very often, in what follows, we shall encounter subspaces that are not necessarily closed; to make them stand out, we shall refer to them as vector subspaces.

\begin{defn} A vector subspace $\cl M$ of $H^2(\bb D, \cl K)$ is said to be nearly invariant under the backward shift $S^*$ if $S^*f\in \cl M$ whenever $f\in \cl M$ 
and $f(0)=0.$
\end{defn}

In light of the fact that $S^*$ is a left inverse of $S$, the definition of nearly invariant under $S^*$ is equivalent to saying $f\in \cl M$ 
whenever $Sf\in \cl M$. This motivated Erard in \cite{Era} to extend the notion of nearly invariant to the setting of bounded below multiplication operators on 
reproducing kernel Hilbert spaces. Before giving Erard's version, first, we provide the following relevant definitions.

\begin{defn} A set of complex-valued functions on a set $X$ is called a reproducing kernel Hilbert space (RKHS) if
\begin{enumerate}
\item $\cl H$ is a vector space with respect to pointwise addition and scalar multiplication;
\item $\cl H$ has a norm with which it is a Hilbert space;
\item  for each fixed $x\in X,$ the point evaluation map $f\mapsto f(x)$ is continuous on $\cl H.$
\end{enumerate}
\end{defn}

Suppose $\cl H$ is an RKHS on a set $X$. Further, suppose $\phi$ is a function on $X$ which multiplies $\cl H$ into itself. Let $M_\phi$ denote this multiplication map. Then it can be seen that $M_\phi$ is linear and bounded. The following is Erard's analouge of nearly invariant in the context of a multiplication operator on an RKHS.  

\begin{defn} Let $\cl H$ be an RKHS on a set $X$ and suppose $M_\phi$ is a multiplication operator on $\cl H$ which is bounded below. Then a vector subspace $\cl M$ of $\cl H$ is 
said to be nearly invariant under division by $\phi$ if $\phi f\in \cl M$ implies $f\in \cl M.$
\end{defn} 

Note that when $M_\phi$ is bounded below on $\cl H,$ then nearly invariant under a left inverse of $\cl M_\phi,$ as we discussed above, would mean $f\in \cl M$ whenever $\phi f\in \cl M$ which clearly justifies Erard's choice for the terminology ``nearly invariant under division by $\phi$". Moreover, the advantage of this terminology is that it brings to light the essence of the definition for multiplication operators. Also, the absence of an explicit mention of a left inverse makes the definition much simpler to follow. 

The following are straightforward but yet important observations.

\begin{lemma}\label{ni1} Let $\cl W$ be an open subset of the complex plane, $\cl H$ be an RKHS on $\cl W$ consisting of analytic functions on $\cl W,$ and let $\phi$ be an analytic function on 
$\cl W$ that multiplies $\cl H$ into itself. If $\phi$ vanishes at a point in $\cl W,$ then the only vector subspace of $\cl H$ that is nearly invariant under division by $\phi$ and contained in $\phi\cl H$ is the zero subspace.     
\end{lemma}

\begin{proof} Suppose $\mathcal{M}$ is a vector subspace of $\cl H$ that is nearly invariant under division by $\phi$ and it is contained in $\phi \mathcal{H}$. Then for $h \in \mathcal{M}$, $h=\phi f$ for some $ f \in \mathcal{H}$. 
Since $\mathcal{M}$ is nearly invariant under division by $\phi$, therefore $f \in \mathcal{M}$.  Again, using the fact that $\cl M$ is contained in $\phi \cl H$ and it is nearly invariant under division by $\phi$, we conclude $f=\phi f_1$ for some $f_1 \in \cl M.$ Then $h=\phi^{2}f_1$. Continuing in the similar fashion, we obtain that for each $n, \ h=\phi^{n+1}f_n$, for some $f_n \in \cl M$. Now since $\phi$ has a zero in $\cl W,$ say at $z_0$, we deduce that the analytic function $h$ has a zero at $z_0$ of every order. This implies that $h=0.$ Hence $\mathcal{M}=\{0\};$ this completes the proof.    
\end{proof}

\begin{lemma}
Let $\mathcal{M}$ be a non-zero Hilbert space contractively contained as a vector subspace in $\mathcal{H}.$ If $\cl R$ is closed in $\cl H,$ then $\cl M\cap \cl R$ is closed in $\cl M.$   
\end{lemma}

\begin{proof} Let $\{h_n\}_{n=0}^{\infty}$ be a sequence in $\mathcal{M}\cap \mathcal{R}$ that converges to 
$h$ in $\mathcal{M}$. Since $\mathcal{M}$ is contractively contained in $\mathcal{H}$, therefore $\{h_n\}$ converges to $h$ in $\cl H.$ But each 
$h_n\in \cl R$ and $\cl R$ is closed in $\cl H.$ Therefore, $h \in  \mathcal{R}$ which implies that 
$h \in \mathcal{M}\cap\mathcal{R}$. Thus   $\cl M\cap \cl R$ is closed in $\cl M.$ 
\end{proof}

We end this section by recalling a few more terminologies. If $\phi$ is a bounded analytic function on $\bb D$, then it multiplies $H^2(\bb D)$ into itself, and 
in this case, the multiplication operator $M_{\phi}$ is a particular example of a Toeplitz operator which is generally denoted as $T_\phi$. Further, a bounded analytic function on $\bb D$ is said to be an inner function if $\lim\limits_{r\to 1^{-}}|\phi(re^{it})|=1$ a.e.. If $\phi$ is an inner function on $\bb D$ and 
$\phi(0)=0,$ then the composition operator, denoted as $C_\phi$, is an isometry on $H^2(\bb D).$ Suppose $\cl K$ is a Hilbert space with an orthonormal basis indexed by a set $I.$ Then direct sum  of $T_\phi$ and $C_\phi$ on $H^2(\bb D, \cl K)$(identified as $\ell^2$-direct sum of $I$-copies of $H^2(\bb D)$ are again bounded operator which we shall denote again by $T_\phi$ and $C_\phi$. 

\section{nearly invariant brangesian subspaces for the backward shift on vector-valued Hardy spaces}\label{Hardy}
In \cite{CCP}, Chalendar, Chevrot, and Partington extended Hitt's theorem (Theorem \ref{hitt}) to the vector-valued setting. They described subspaces of $H^2(\bb D, \bb C^n)$ that are nearly invariant under the backward shift operator $S^*$ on $H^2(\bb D, \bb C^n)$. In this section, we investigate their result in the de Branges setting. We describe Hilbert spaces contractively contained in $H^2(\bb D, \bb C^n)$ and nearly invariant under $S^*.$ It is crucial to note that we do not assume these vector subspaces to be closed in the Hardy space.  

Before presenting their description, we explain some notations essential to understanding their result and which will also be used throughout this section. Suppose $g_1, \dots, g_m$ be 
$\bb C^n$-valued 
functions on $\bb D.$ Let $G$ denote $n\times m$ matrix-valued function that maps $z\in \bb D$ to $n\times m$ matrix with column vectors $g_1(z), \dots, g_m(z).$ Now suppose $f$ is a 
$\bb C^m$-valued function on $\bb D$. Clearly, we can write 
$f=(f_1,\dots, f_m)$, where $f_1, \dots, f_m$ are scalar-valued functions on $\bb D$. Then for each $z\in \bb D,$ the matrix multiplication $G(z)f(z) =(g_1(z) \cdots g_m(z)) \begin{pmatrix} f_1(z) \\ \vdots \\ f_m(z)\end{pmatrix}=\sum_{i=1}^m f_i(z)g_i(z)$ is well-defined. We shall use $Gf$ to denote the function $z\mapsto G(z)f(z)=\sum_{i=1}^mf_i(z)g_i(z)$.  Clearly, if each $g_i$ and $f$ are analytic on $\bb D$, then $Gf$ is also analytic on $\bb D$. 

\begin{thm}\label{CCP}(\cite[Theorem 4.4]{CCP}) Let $\cl F$ be a nearly $S^*$-invariant subspace of $H^2(\bb D, \bb C^n)$ and let $\{g_1, \dots, g_r\}$ be an orthonormal basis of $\cl M\ominus (\cl M\cap zH^2(\bb D, \bb C^n))$. Let $G$ be the $n\times r$ matrix-valued function with columns $g_1, \dots, g_r$. Then there exists an isometric mapping 
$$
\cl J: \cl F \to {\cl F}^{'} \ {\rm given \ by} \ Gf\mapsto f, 
$$
where ${\cl F}^{'}:=\{f\in H^2(\bb D, \bb C^r): \exists \ h\in \cl F, \ h=Gf\}.$ Moreover, $\cl F^{'}$ is subspace of $H^2(\bb D, \bb C^r)$ that is $S^*$-invariant.
\end{thm} 

We shall now present the main result (Theorem \ref{MTH}) of this section. It is an analogue of Theorem \ref{CCP} for the de Branges setting. We start with the following preliminary observation.

\begin{prop}\label{ni2} Let $\mathcal{M}$ be a non-zero Hilbert space contractively contained in the Hardy space $H^2(\bb D, \bb C^n).$ Suppose $\cl M$ is nearly invariant under the backward shift on $H^2(\bb D, \bb C^n).$ Then $\cl M\ominus (\cl M\cap zH^2(\bb D, \bb C^n))$ is non-zero and its dimension can be at most $n$. 
\end{prop}

\begin{proof} First note that since $\cl M$ is non-zero, therefore Lemma \ref{ni1} implies that $\cl M$ cannot be contained in $zH^2(\bb D, \bb C^n).$  

Now as $\mathcal{M}$ is contractively contained in $\mathcal{H}^2(\mathbb{D}, \mathbb{C}^n)$, for each $w\in \mathbb{D},$ the point evaluation map 
$ E_w : \mathcal{M} \longrightarrow \mathbb{C}^n$ given by $E_w(f)=f(w)$ 
is bounded. Let $\{e_1,\dots, e_n\}$ be the canonical orthonormal basis of $\bb C^n.$ We claim that the set $\{ g_i:=E^*_0(e_i) : 1 \le i \le n\}$ spans $\mathcal{M}\ominus(\mathcal{M}\cap zH^2(\mathbb{D}, \mathbb{C}^n))$. 
For any $f$ in $\mathcal{M}\cap zH^2(\mathbb{D}, \mathbb{C}^n)$, 
$$
\langle g_i, f \rangle_\mathcal{M}
= \langle e_i, E_0(f) \rangle_{\bb C^n}
=\langle e_i, f(0) \rangle_{\bb C^n}=\langle e_i,0 \rangle_{\bb C^n}=0.
$$
Thus each $g_i$ belongs to $\cl M\ominus (\cl M\cap zH^2(\bb D, \bb C^n)).$ Further, if $f\in \cl M$ is orthogonal to each $g_i$. Then $f(0)$ is orthogonal to 
$e_i$ for each $1 \le i \le n$. This forces $f(0)=0$ which means 
$f \in \cl M\cap zH^2(\mathbb{D}, \mathbb{C}^n)$. Hence, the set $\{g_i: 1\le i\le n\}$ spans $\mathcal{M}\ominus(\mathcal{M}\cap zH^2(\mathbb{D}, \mathbb{C}^n))$. This completes the proof. 
\end{proof}

\begin{thm}\label{MTH} Let $\mathcal{M}$ be a non-zero Hilbert space contractively contained in $H^2(\mathbb{D},\mathbb{C}^n).$ Suppose $\cl M$ is nearly invariant under $S^*$ and $\|z h\|_\mathcal{M}\ge \|h\|_\cl M$ 
whenever $zh \in \mathcal{M}$. If $\{g_1,\dots, g_r\}, \ 1\le r\le n$, is an orthonormal basis for $\cl M\ominus (\cl M\cap zH^2(\bb D, \bb C^n)),$ then there exists a vector subspace $\mathcal{N}$ of $H^2(\mathbb{D},\mathbb{C}^r)$ that is $S^*$-invariant such that $\cl M$ is in one-to-one correspondence with $\cl N$ via the linear map 
$$
\cl G: \cl N \to \cl M \quad  {\rm given \ by} \quad f\mapsto Gf,
$$
where $G$ is the matrix-valued function whose columns are $g_1, \dots, g_r.$ Moreover, for each $h\in \cl M, \ ||h||_\cl M \ge ||f||_2,$ where $h=Gf$ with $f\in \cl N.$
\end{thm}

\begin{proof} Firstly, using Proposition \ref{ni2}, $\cl M\cap zH^2(\bb D, \bb C^n )$ is closed in $\cl M,  \ \cl M\ominus (\cl M\cap zH^2(\bb D, \bb C^n))$ is non-zero, 
and dimension of $\cl M\ominus (\cl M\cap zH^2(\bb D, \bb C^n))$ is at most $n.$ 
Let $P$ denote the orthogonal projection of $\cl M$ onto $\cl M\cap zH^2(\bb D, \bb C^n)$ and let $Q=I_{\cl M}-P.$ 

Now since $\cl M$ is neary invariant under $S^*,$ therefore $S^*P$ is a well-defined linear mapping of $\cl M$ into itself. Let us define 
$$
R=S^*P.
$$ 
Then the hypothesis $||f||_{\cl M} \le ||zf||_{\cl M}$ whenever $zf\in \cl M$ implies that $R$ is a contraction on $\cl M$.   

Fix any $h\in \cl M$. We decompose it as $h=Qh+Ph.$ Note that $Ph\in zH^2(\bb D, \bb C^n).$ This means $Ph=SS^*Ph=SRh.$  Then we have 
\begin{equation}\label{MTH1}
h=Qh+SRh
\end{equation}
and 
\begin{equation}\label{MTH2}
||Qh||_{\cl M}^2+||Rh||_{\cl M}^2\le ||Qh||_{\cl M}^2+||SRh||_{\cl M}^2 = ||h||_{\cl M}^2.
\end{equation}

As $\{g_1,\dots, g_r\}$ is an orthonormal basis for 
$\cl M\ominus (\cl M\cap zH^2(\bb D, \bb C^n)),$ therefore we can write  
$$
Qh=a_{01}g_1 + a_{02}g_2 +...+a_{0r}g_r = G A_0,
$$
where $a_{01}, \dots, a_{0r}$ are scalars, $A_0=\begin{pmatrix} a_{01} \\ \vdots \\ a_{0r}\end{pmatrix}$ $\in \mathbb{C}^r$, and $G$ is the 
$n\times r$ matrix-vaued function on $\bb D$ with columns $g_1, \dots, g_r$. Thus,   
\begin{equation}\label{MTH3}
h=GA_0+ SRh. 
\end{equation}
Further, $||Qh||^2_{\cl M}=\sum_{i=1}^r |a_{0i}|^2=||A_0||^2_{\bb C^r}$. Thus, Inequality (\ref{MTH2}) yields 
\begin{equation}\label{MTH4}
||A_{0}||^2_{\bb C^r} +||Rh||_{\cl M}^2\le ||h||_{\cl M}^2.
\end{equation}

Now $Rh\in \cl M.$ Then repeating the above arguments for $Rh$ in place of $h$, we obtain a vector $A_1\in \bb C^r$ such that 
$$
Rh=GA_1+SR^2h
$$
 and $||A_1||^2_{\bb C^r}+||R^2h||^2_{\cl M}\le ||Rh||^2_{\cl M}.$ Then Equations (\ref{MTH3}) and (\ref{MTH4}) yields 
\begin{equation}\label{MTH5}
h=GA_0+ G(zA_1) + S^2R^2h
\end{equation}
and 
\begin{equation}\label{MTH6}
||A_{0}||^2_{\bb C^r} +||A_1||^2_{\bb C^r} + ||R^2h||_{\cl M}^2\le ||h||_{\cl M}^2.
\end{equation}

Again, $R^2h\in \cl M.$ Continuing as above, we obtain a sequence $\{A_n\}$ in $\bb C^r$ such that for each postive integer m 
\begin{equation}\label{MTH7}
h=G(A_0+ A_1z+\cdots+ A_m z^m) + S^{m+1}R^{m+1}h
\end{equation}
and 
\begin{equation}\label{MTH8}
\sum_{i=0}^m ||A_i||^2_{\bb C^r} + ||R^{m+1}h||_{\cl M}^2\le ||h||_{\cl M}^2.
\end{equation}
   
The Inequality (\ref{MTH8}) establishes that  
$$
\sum_{n=0}^{\infty}\|A_n\|^2_{\bb C^r} < \infty. 
$$
Thus, 
$$
f(z)=\sum_{m=0}^{\infty}A_mz^m
$$ 
belongs to $H^2(\mathbb{D},\mathbb{C}^r)$.

Cearly, $Gf$ is analytic on $\bb D.$ Now comparing the coefficient of $z^m$ in $h$, $Gf$, and using Equation (\ref{MTH7}), we conclude, 
$h=Gf.$ Also, using Equation (\ref{MTH8}), 
$$
||f||_2\le ||h||_{\cl M}.
$$
Hence, for each $h\in \cl M$ there exists an $f\in H^2(\bb D, \bb C^r)$ such that $h=G f$ and $||f||_2\le ||h||_{\cl M}.$  
 
Let $\cl N=\{f\in H^2(\bb D, \bb C^r): Gf\in \cl M\}.$ Then $\cl N$ is a vector subspace of $H^2(\bb D, \bb C^r)$. Clearly, the mapping $\cl G:\cl N \to \cl M$ given by 
$\cl G(f)=Gf$ is a well-defined surjective linear map. To show it is one-to-one, let $Gf=0.$ Suppose $f(z)=\sum_{m=0}^\infty A_mz^m.$ Write $f=A_0+zf_1,$ where 
$f_1(z)=\sum_{m=0}^\infty A_{m+1}z^m.$ Then $Gf=GA_0+G(zf_1)$ and 
$zGf_1\in \cl M\cap zH^2(\bb D,\bb C^n).$ This implies that $GA_0=Q(Gf)=0$, which further implies that $A_0=0$; hence $f=zf_1.$ Thus, $Gf_1=0. $ Continuing in this way, we can show that $A_k=0$ for all $k.$ 
Hence, $f=0.$ Thus $\cl G$ is one-to-one.  

Finally we shall show that    
$\cl N$ is invariant under $S^*.$ 
Let $f=\sum_{m=0}^\infty A_m z^m\in \cl N.$ Then there exists an $h\in \cl M$ such that $h=Gf$. Now,   
$$
    h=Gf=Q(Gf)+S R(Gf).
$$
But $Q (Gf)=GA_0.$ Therefore
\begin{align*}
    h=GA_0+S R(Gf)
\end{align*}
which implies
$$
SR(Gf)=G(f-A_0)=G \Bigg(\sum_{k=1}^{\infty}A_k z^{k} \Bigg),
$$
and hence 
$$
R(Gf)=G\Bigg(\sum_{k=1}^{\infty}A_k z^{k-1}\Bigg)=G\Bigg(S^*\Bigg(\sum_{k=0}^{\infty}A_k z^{k}\Bigg)\Bigg)=G(S^*f).
$$
Since $R(Gf) \in \mathcal{M}$, therefore by definition, $S^*f\in \mathcal{N}$. Hence $\mathcal{N}$ is invariant under $S^*$. This completes the proof.
\end{proof}

\begin{remark}\label{MRH} Note that the Hitt's description of a nearly $S^*$-invariant subspace of $H^2(\bb D)$ (Theorem \ref{hitt}) as well as it's vectorial generalization (Theorem \ref{CCP}) both have three parts to them, namely, the represenation in terms of $S^*$-invariant subspace, the norm preservation between nearly $S^*$-invariant subspace and the corresponding $S^*$-invariant subspace, and the closedness of the $S^*$-invariant subspace. Now the description (Theorem \ref{MTH}) we obtain for our setting does gives a represenation that is similar to the one given in 
Theorem \ref{hitt} for the scalar case and Theorem \ref{CCP} for the vector case, but our description in the general case neither gaurantees the preservation of norm nor does it gaurantees the closedness of the $S^*$-invariant vector subspace. Interestingly, with the help of the following two examples we show that either of these can't be  promised for our setting in general.
\end{remark}

\begin{ex}\label{EH1}(Failure of equality of norms). Let $\cl M=span\{1+z,z+z^2\}$ and $\mathcal{U} : H^2(\mathbb{D}) \longrightarrow H^2(\mathbb{D})$ be the linear operator given by 
$$
\cl U(1)=1, \quad \cl U (z)=\sqrt{2}z, \quad \cl U(z^2)=\sqrt{2}z^2, \quad \cl U(z^n)=z^n \ {\rm for} \  n\ge 3.
   $$
Define a norm $\|.\|_\mathcal{M}$ on $\mathcal{M}$ by
\begin{align*}
    \|f\|_\mathcal{M}:=\|\mathcal{U}f\|_2 \quad {\rm for} \ f\in \mathcal{M}
\end{align*}
Then $\mathcal{M}$ equipped with norm $\|\cdot\|_\mathcal{M}$ is a Hilbert space contractively contained in $H^2(\mathbb{D})$ that is nearly $S^*$-invariant.  

Clearly, $\cl M\cap zH^2(\bb D)=span\{z+z^2\}.$ Let $f\in \cl M\cap zH^2(\bb D).$ Then $f=\alpha (z+z^2)$ for some scalar $\alpha$ and 
$||S^*f||_{\cl M}=|\alpha ||\cl U(1+z)||_2=|\alpha|\sqrt{3}.$ On the other hand, $||f||_{\cl M}=|\alpha| ||\cl U(z+z^2)||_{\cl M}=|\alpha|2.$ Therefore, 
$||S^*(f)||_{\cl M}\le ||f||_{\cl M}$ for each $f\in \cl M\cap zH^2(\bb D)$ which simply means that $||zh||_{\cl M}\ge ||h||_{\cl M}$ whenever $zh\in \cl M.$
Thus, $\cl M$ satisfy the hypotheses of Theorem \ref{MTH}.

Now we can verify that 
$$
\cl M\ominus (\cl M\cap zH^2(\bb D)) = span\{g\},
$$
where $g(z)=\frac{(z-2)(z+1)}{2\sqrt{2}}$ and $||g||_{\cl M}=1.$  

Then using Theorem \ref{MTH}, there exists a $S^*$-invariant vector subspace $\cl N$ of $H^2(\bb D)$ such that $\cl M=g\cl N.$ 

For $1+z \in \mathcal{M}$, we have $1+z=g\frac{2\sqrt{2}}{z-2}$. Therefore,  $f=\frac{2\sqrt{2}}{z-2}\in \mathcal{N}$.  
Notice that 
\begin{align*}
\|z+1\|^2_\mathcal{M}
&=\|\sqrt{2}z+1\|_2^2\\
&=3
\end{align*}
and 
\begin{eqnarray*}
\|f\|_2^2&=&\Big\|\frac{2\sqrt{2}}{z-2}\Big\|_2^2\\
&=& 8\Big\|\frac{1}{z-2}\Big\|_2^2\\
&=& \frac{8}{3}.
\end{eqnarray*}
Hence, $1+z=gf$ and $\|1+z|\|_{\cl M}>\|f\|_2.$ 
\end{ex}

\begin{ex}\label{EH2}(Failure of the closedness). Let $\cl D$ denote the classical Dirichlet space consisting of analytic functions on the unit disc $\bb D$ with the norm 
$||f||_{\cl D}^2 = \sum_{i=0}^\infty |a_i|^2(i+1)$ for $f(z)=\sum_{i=0}^\infty a_iz^i\in \cl D.$ Recall that $\cl D$ is a Hilbert space contractively contained in 
$H^2(\bb D),$ and it is not closed in $H^2(\bb D).$  

Let $\theta$ be a bounded analytic function on $\bb D$ with $||\theta||_\infty=1$ and $\theta(0)>0.$ Set 
$$
\cl M = \theta \cl D
$$
and define $||\theta f||_{\cl M}=||f||_{\cl D}.$ Clearly, $\cl M$ is a vector subspace of $H^2(\bb D),$ $||\cdot||_{\cl M}$ a norm on $\cl M$ with respect to which $\cl M$ becomes a Hilbert 
space contractively contained in $H^2(\bb D).$  

Let $f\in \cl M\cap zH^2(\bb D).$ Then $f=\theta h$ for some $h\in \cl D$ and $f(0)=0.$ Thus, $h(0)=0$ because $\theta(0)>0.$ This implies that $h=zh_1$ for some 
$h_1\in H^2(\bb D).$ But $zh_1\in \cl D$ implies $h_1\in \cl D$ and $||zh_1||_{\cl D}\ge ||h_1||_{\cl D}.$ Therefore, 
$S^*(f)=\theta h_1\in \cl M$ and $||S^*f||_{\cl M}=||h_1||_{\cl D}\le ||zh_1||_{\cl D} = ||f||_{\cl M}.$ This means $\cl M$ is nearly $S^*$-invariant and $||zg||_{\cl M}\ge ||g||_{\cl M}$ whenever $zg\in \cl M.$ 
Hence, $\cl M$ satisfy the hypotheses of Theorem \ref{MTH}. Note that $\cl M\ominus (\cl M\cap zH^2(\bb D))=span\{\theta\}.$ Therefore there exists an $S^*$-invariant vector subspace $\cl N$ of $H^2(\bb D)$ such that $\cl M=\theta \cl N.$ But this simply means $\cl N$ equals $\cl D$, which is not closed in $H^2(\bb D).$ Hence, this examples shows that the $S^*$-invariant vector subspace we obtain in the represenation given by Theorem \ref{MTH} may not be closed. 
\end{ex}

\section{nearly invariant brangesian subspaces related to multiplication operators on reproducing kernel Hilbert spaces}\label{rkhs} 

In \cite{Era}, Erard extended the study of nearly invariant subspaces on $H^2(\bb D)$ to reproducing kernel Hilbert spaces. 
 
 \begin{thm}(Erard \cite[Theorem 5.1]{Era})\label{Era} Let $\cl H$ be an RKHS consisting of complex-valued analytic functions on $\bb D$ on which multiplication with $z$ is 
well-defined with dimension of $\cl H\ominus z\cl H$ equals $1$ and $||h||_{\cl H}\le ||zh||_{\cl H}$ for all $h\in \cl H.$ Assume also that there exists $f\in \cl H$ with 
$f(0)\ne 0.$ Let $\cl M$ be a non-zero subspace of $\cl H$ which is nearly invariant under the backward shift. Let $g$ be any unit vector of $\cl M\ominus (\cl M\cap z\cl H).$ Then there exists a linear subspace $\cl N$ of $H^2(\bb D)$ such that 
$$   
\cl M= g\cl N \ \  {\rm and} \ ||h||_{\cl H}\ge ||\frac{h}{g}||_2.
$$
Besides, $\cl N$ is invariant under the backward shift and $g(0)\ne 0.$
\end{thm}

Our main result (Theorem \ref{MTRKHS}) of this section generalizes Erard's Theorem. We describe Hilbert spaces that are contractively contained in an RKHS of analytic functions on the unit disc which are nearly invariant under division by an inner function. So, in our result, subspaces have been replaced with contractively contained 
Hilbert spaces and multiplication with $z$ has been replaced with an inner function.
 
Before proceeding further, we would like to compare Erard's theorem with Hitt's. Erard's theorem replaces $H^2(\bb D)$ by a much general RKHS, and 
also, instead of assuming $M_z$ to be an isometry, it only assumes it to be bounded below. However, the drawback of Erard's theorem is that although the representation of a nearly $S^*$-invariant subspace when $\cl H=H^2(\bb D)$ is very similar to what Hitt's theorem gives, it does not infer the correspondence between a nearly $S^*$-invariant 
subspace and the corresponding $S^*$-invariant vector subspace to be an isometry, and in fact, it doesn't even guarantee the closedness of the $S^*$-invariant vector subspace.  

Interestingly, we can deduce our Theorem \ref{MTH} (the scalar case) as a corollary from Erard's Theorem without missing any detail because we have shown, with  Examples \ref{EH1} and \ref{EH2}, that the two features of the description of a nearly $S^*$-invariant subspaces that Erard's theorem misses do not hold for our setting in general. 

We first prove the following analogue of Lemma 2.1 from \cite{Era} that played a pivotal role in proving Erard's Theorem. Indeed, we have proved this result (in disguise) within the proof of Theorem \ref{MTH}, and we need it again for Theorem \ref{MTRKHS}. We feel that it is a crucial observation and is interesting in its own right; so, we are proving it here as a separate result.

\begin{lemma}\label{rkhs-l} Let $T$ be a bounded operator on a Hilbert space $\mathcal{H}$  such that $\|Th\|_\mathcal{H}\ge \|h\|_\mathcal{H}$ for all $h$ in $\mathcal{H}$. Let 
$\mathcal{M}$ be a Hilbert space contractively contained in $\cl H$ such that $h\in \cl M$ whenever $Th\in \cl M$ and $\|Th\|_\mathcal{M}\ge \|h\|_\mathcal{M}$. If 
$P$ denotes the orthogonal projection of $\cl M$ onto $\cl M\cap T{\cl H}$ and $Q= I_{\cl M}- P,$ then $R:= (TT^*)^{-1}T^*P$ is a 
well-defined contraction on $\cl M$, and for every positive integer $m$, we can decompose each $h\in \cl M$ as 
$$
h=\sum_{k=0}^{m}T^kQR^kh + T^{m+1}R^{m+1}h 
$$
and 
$$
\|h\|^2_\mathcal{M}\ge \sum_{k=0}^m\|QR^kh\|^2_\mathcal{M}.
$$
\end{lemma}

\begin{proof} Let $h\in \cl M.$ Then 
\begin{eqnarray*}
TRh&=&T(TT^*)^{-1}T^*P(h)\\
&=&T(TT^*)^{-1}T^*Th_0 \ \ (Ph =T h_0 \text{ for some} \ h_0\in \mathcal{H})\\
&=&Th_0\\
&=&P(h)
\end{eqnarray*}
This shows $TRh\in \cl M$, but then $Rh\in \cl M$ and $||TRh||_{\cl M}\ge ||Rh||_{\cl M}.$  Thus, $||Rh||_{\cl M}\le ||TRh||_{\cl M}= ||Ph||_{\cl M}\le ||h||_{\cl M}$. 
Therefore, $R$ is a well-defined contraction on $\mathcal{M}$.  

Again, let $h$ in $\mathcal{M}$ and decompose it as   
\begin{equation}\label{rkhs1}
h=Qh + Ph=Qh + TRh.    
\end{equation}

Then 
\begin{equation}\label{rkhs2}
\|h\|^2_\mathcal{M}= \|Qh\|^2_\mathcal{M} + \|TRh\|^2_\mathcal{M} \ge \|Qh\|^2_\mathcal{M} + \|Rh\|^2_\mathcal{M}.
\end{equation}
Now $Rh \in \mathcal{M}$. Thus,
\begin{equation}\label{rkhs3}
Rh= QRh + TR^2h    
\end{equation}
and 
\begin{equation}\label{rkhs4}
||Rh||_{\cl M}^2\ge ||QRh||_{\cl M}^2 + ||R^2h||_{\cl M}^2.
\end{equation}

Then, using Inequalities (\ref{rkhs1}) - (\ref{rkhs4}), we have 
$$h= Qh + TQRh + T^2R^2h $$
and 
$$
\|h\|^2_\mathcal{M}\ge \|Qh\|^2_\mathcal{M} + \|QRh\|^2_\mathcal{M}+ \|R^2h\|^2_\mathcal{M} 
$$

Continuing this process, we obtain that for non-negative integer $m$, we can write 
$$
h=\sum_{k=0}^{m}T^kQR^kh + T^{m+1}R^{m+1}h
$$
and 
$$
\|h\|^2_\mathcal{M}\ge \sum_{k=0}^m\|QR^kh\|^2_\mathcal{M}.
$$
This completes the proof.
\end{proof}

\begin{thm}\label{MTRKHS} Let $\mathcal{H}$ be an RKHS consisting of analytic functions on $\bb D.$ Let $\phi$ be an inner function such that $\phi(0)=0, \ \phi \cl H\subseteq \cl H,$ and $||h||\le ||\phi h||$ for every $h\in \cl H$. 
Assume that if $\phi h\in \cl H$ for an analytic function $h$ on $\bb D,$ then $h\in \cl H.$
Let $\cl M$ be a Hilbert space contractively contained in $\cl H$ which is nearly invariant under division by $\phi$ and $||\phi h||_{\cl M}\ge ||h||_{\cl M}$ whenever $\phi h\in \cl M.$ Then there exists a vector subspace $\mathcal{N}$ of $H^2(\mathbb{D},\ell^2(I))$ invariant under $T_{\phi}^*$ such that $\cl M$ is in one-to-one correspondence with $\cl N$ via the linear map 
$$
\cl G: \cl N\to \cl M \quad {\rm given \ by} \quad (Gf)(z) = \sum_{i\in I}g_i(z)f_i(z) \ \ (\rm pointwise),
$$ 
where $f=(f_i)_{i\in I}$ and $\{g_i:i\in I\}$ is an orthonormal basis of $\cl M\ominus (\cl M\cap \phi \cl H).$ Moreover, $\|h\|_\mathcal{M}\ge \|f\|_{H^2(\mathbb{D},\ell^2(I))}$ if $h(z)=\sum_{i\in I}g_i(z)f_i(z)$ for $f=(f_i)_{i\in I}\in \cl N.$  

\end{thm}

\begin{proof} Let $P$ denote the orthogonal projection of $\cl M$ onto it's closed subspace $\cl M\cap \phi \cl H$ and $Q=I_{\cl M} - P.$ 
Let $h\in \cl M$. Then using Lemma \ref{rkhs-l}, 
\begin{equation}\label{rkhs5}
 h=\sum_{k=0}^mM_{\phi}^m QR^m h+M^{m+1}_{\phi}R^{m+1}h \ \ \ {\rm for \ every} \ m\ge 0,  
\end{equation}
and 
\begin{equation}\label{rkhs6}
\sum_{m=0}^\infty ||Q_{\cl M}R^mh||_{\cl M}^2\le ||h||_{\cl M}^2,
\end{equation}
where $R:=(M_{\phi}M_{\phi}^*)^{-1}M_{\phi}P$ is a contraction on $\cl M.$ 

Since $\{g_i: i\in I\}$ is an orthonormal basis of $\cl M\ominus (\cl M\cap \phi \cl H)$, therefore for every $k\ge 0,$ we have  
$$
 QR^{k}h=\sum_{i \in I}c_{ki}g_i 
 $$
 for some $\{c_{ki}\}_{i\in I}\in \ell^2(I).$ 
 
 Then 
$$
 h=\sum_{k=0}^{m}\sum_{i \in I}c_{ki}M^{k}_{\phi }g_{i} + M^{m+1}_{\phi}R^{m+1}h
$$
and 
$$
 \sum_{i \in I}\sum_{k=0}^{\infty}  |c_{ki}|^{2}\le \|h\|^{2}_{\mathcal{M}}.
$$
Thus for every $i \in  I, \ q_i(z):=\sum_{k=0}^{\infty}c_{ki}z^k$ is in $H^2(\bb D).$ Further, since $\phi$ is an inner function with $\phi(0)=0,$ therefore the composition operator $C_{\phi}$ induced by $\phi$ is an isometry on $H^2(\bb D).$ Thus, $f_{i}=C_{\phi}(q_i)=\sum_{k=0}^\infty c_{ki}\phi^k$ belongs to $C_\phi(H^2(\bb D))$ and 
$||f_i||_2^2=\sum_{k=0}^\infty |c_{ki}|^2.$  

Then, for any $w \in \mathbb{D}$, 
\begin{eqnarray*}
&& \sum_{i \in I}|(g_i f_i)(w)| \\&\le & \Bigg(\sum_{i \in I}|g_i(w)|^2\Bigg)^{1/2}\Bigg(\sum_{i \in I}|f_i(w)|^2\Bigg)^{1/2}\\
 & \le& \Bigg(\sum_{i \in I}|<g_i,k_w>_\mathcal{M}|^2\Bigg)^{1/2}\Bigg(\sum_{i \in I}\Bigg(\sum_{k=0}^{\infty}|c_{ki}|^2\Bigg)\Bigg(\sum_{k=0}^{\infty}|\phi(w)|^{2k}\Bigg)\Bigg)^{1/2}\\
 &\le& \|Qk_w\|_\mathcal{M}\|h\|_\mathcal{M}\frac{1}{\sqrt{1-|\phi(w)|^2}},
\end{eqnarray*}
where $k_{w}$ is the kernel function of $\cl M$ at the point $w.$ This shows that the series $\sum_{i \in I}g_{i}f_{i}$ converges at each point in $\mathbb{D}$.   

We shall now prove that $\sum_{i \in I}g_{i}f_{i}$ is analytic on $\mathbb{D}$. Suppose $z_0 \in \mathbb{D}$ and choose $r>0$ such that $ \overline{D(z_0,r)}$ $\subset \mathbb{D}$. Let $w \in \overline{D(z_0,r)}$. Since the kernel function $K$ of $\cl M$ is analytic in the first variable and coanalytic in the second variable, therefore $K$ is bounded on compact subsets of $\bb D^2.$ Thus, there exists a constant $A>0$, depending on $z_0$ and $r,$ such that $\|k_w\|_{\cl M}^2=K(w,w)\le A$. Also, $\sup_{|w-z_0|\le r}|\phi(w)| \le B <1$, where $B$ depends on $z_0$ and $r$. Therefore,
$$
\sum_{i \in I}|(g_i f_i)(w)| \le \frac{A}{\sqrt{1-B}}\Bigg( \sum_{i \in I}\sum_{k=0}^{\infty}|c_{ki}|^2\Bigg)^{1/2}
$$
This also implies that $\{i\in I: f_i(z)g_i(z)\ne 0\}$ must be countable which means we can assume the above sum on the left must be a countable sum. Then using 
the Weierstrass M-test the series converges uniformly on $\overline{D(z_0,r)}$. Thus the series $\sum_{i\in I}g_i f_i$ converges locally uniformly on $\mathbb{D}$. Hence 
$\sum_{i\in I}(g_if_i)$ is analytic on 
$\mathbb{D}$.  

Further, using Equation (\ref{rkhs5}),  
$h - \sum_{i \in I}g_{i}f_{i}$ is an analytic function on $\mathbb{D}$ having zero of every order at 0. Hence 
\begin{equation}\label{rkhs7}
    h(z)=\sum_{i \in I}g_{i}(z)f_{i}(z) \ \ \text{ for \ every} \ z\in\mathbb{D},
\end{equation}
$f_i\in C_{\phi}(H^2(\bb D))$ for each $f_i\in I$
and 
$$
\sum_{i\in I}||f_i||_2^2=\sum_{i\in I}\sum_{k=0}^\infty |c_{ki}|^2\le ||h||_{\cl M}^2.
$$  

Now define 
\begin{eqnarray*}
\mathcal{N}=\{f = (f_i)_{i\in I}\in H^{2}(\mathbb{D},\ell^2(I)) &:& f_i\in C_\phi(H^2(\bb D)), \  \exists \ h \in \mathcal{M}, \\
&& h(z)=\sum_{i\in I}g_i(z)f_i(z) \ {\rm for} \ z\in \bb D\}.
\end{eqnarray*}

Clearly $\mathcal{N}$ is a vector subspace of $H^{2}(\mathbb{D},\ell^2(I)$ and the map $\cl G(f)(z)=\sum_{i\in I}g_i(z)f_i(z), z\in \bb D,$ is a well-defined linear surjective map. 
To show it is one-to-one, we shall show that every $f\in \cl N$ is uniquely determined by $\cl Gf.$ Let $f=(f_i)_{i\in I}\in \cl N.$ Then there exists $h\in \cl M$ such that $h=\sum_{i\in I}g_if_i.$ Let $f_i=\sum_{k=0}^\infty a_{ki}\phi^k.$ Then 
$h=\sum_{i\in I}c_{0i}g_i+\phi \left(\sum_{i\in I}g_i\tilde{f_i}\right),$ where $\tilde{f_i}=\sum_{k=0}^\infty c_{(k+1)i}\phi^k.$ Then $\phi\left(\sum_{i\in I}g_i\tilde{f_i}\right)\in \cl H$ 
which yields $\sum_{i\in I}g_i\tilde{f_i}\in \cl H.$ Thus, $h-\sum_{i\in I}c_{0i}g_i\in \cl M\cap \phi \cl H.$ Therefore, $Q(h)=\sum_{i\in I}c_{0i}g_i.$ This means, 
$$
\langle{f_i,1}\rangle=c_{0i}=\langle{Qh,g_i}\rangle.
$$
for each $i.$ Now, $Ph=\phi\left(\sum_{i\in I}g_i\tilde{f_i}\right)$ and $P=M_\phi R.$ Therefore, $Rh=\sum_{i\in I}g_i\tilde{f_i}.$ Again, repeating the above arguments, 
$QRh=\sum_{i\in I}c_{1i}g_i$ which implies 
$$
\langle{f_i,\phi}\rangle=c_{1i}=\langle{QRh,g_i}\rangle.
$$ 

Continuing like this, we ontain 
$$
\langle{f_i,\phi^k}\rangle=\langle{QR^kh,g_i}\rangle.
$$ 
This establishes the claim.

Lastly, we shall show that $\mathcal{N}$ is invariant under $T^*_\phi$. 
Let $f=(f_i)_{i\in I}\in \cl N.$ Then by definition, for each $i, \ f_i\in C_\phi(H^2(\bb D))$ and there exists an $h\in \cl M$ such that $h(z)=\sum_{i\in I}g_i(z)f_i(z)$ for every $z\in \bb D.$ We decompose $h$ as 
$$
h=Qh+ Ph= Q+M_\phi Rh,
$$
since $P=M_{\phi}R.$ Then 
$$
h=\sum_{i\in I}g_if_i = \sum_{i\in I}c_{0i}g_i+M_\phi R(\sum_{i\in I}g_if_i),
$$
where for each $i\in I, \ f_i=\sum_{k=0}^\infty c_{ki}\phi^k$ which implies
$$
  M_\phi R(h)=\sum_{i\in I}g_i(f_i-c_{0i})
$$
and therefore 
$$
 R(h)=\sum_{i\in I}g_iT_{\phi}^*(f_i).
$$
Hence, $T^*_\phi(f)=(T_{\phi}^* f_i)_{i\in I} \in \mathcal{N}$ which establishes that $\mathcal{N}$ is invariant under $T^*_\phi$.
\end{proof}

\begin{remark} 
In \cite{LP}, Liang and Partington used Erard's methods from \cite{Era} to describe subspaces of Dirichet-type spaces $\cl D_{\alpha} (-1\le \alpha \le 1)$ that are nearly invariant under division by a finite Blaschke factor. For $\alpha\ge 0,$ $\cl D_{\alpha}$ posses an equivalent norm with respect to which $M_{\phi}$ is bounded below on it with a lower bound 1. Hence, our Theorem extends Theorem 3.4 from \cite{LP} to a vastly general situation.
\end{remark}

\section{nearly invariant Brangesian subspaces with finite defect related to multipication operators on reproducing kernel Hilbert spaces}\label{finite}
Chalendar, Gallarado-Gutierrez, and Partington introduced and studied the notion of nearly $S^*$-invariant subspaces of $H^2(\bb D)$ with finite defect in \cite{CGP}. 
In this Section we shall extend our work from Section \ref{rkhs} to the finite defect case. This extension is motivated by work of Chattopadhyay and Das from \cite{CD}. They, following Erard's ideas, as discussed in Section 4, extended Liang and Partington's description \cite{LP} of nearly $S^*$-invariant subspaces of Dirichlet-type spaces to the finite defect situation. We shall first introduced some definitions and terminologies that we need in this section. 

Let $\cl H$ be a Hilbert space. Suppose $\cl M$ is a vector subspace of $\cl H$ which is a Hilbert space (with maybe a different norm) and $\cl F$ is a closed subspace of 
$\cl H$ such that $\cl M\cap \cl F=\{0\}.$ Then the vector subspace $\cl M+\cl F$ of $\cl H$ becomes a Hilbert space with respect to the norm given by  
$$
||h+f||_{\oplus}^2=||h||^2_{\cl M}+||f||_{\cl H}^2; \ \ h\in \cl M,  \ f\in \cl F. 
$$
Furthermore, $\cl M$ and $\cl F$ are closed orthogonal subspaces of $(\cl M+\cl F, ||\cdot||_{\oplus})$. Henceforth, we shall use  
$\cl M\oplus \cl F$ to denote $(\cl M+\cl F, ||\cdot||_{\oplus})$.  

The following result is an anlogue of Lemma 2.1 from \cite{Era} and our Lemma \ref{rkhs-l} for the finite defect situation. 

\begin{lemma}\label{defect-l} Let $\cl H$ be a Hilbert space and $T\in B(\cl H)$ with $||Th||_{\cl H}\ge ||h||_{\cl H}$ for all $h\in \cl H.$ Let 
$\mathcal{M}$ be a Hilbert space contractively contained in $\cl H$ for which there exists a finite dimensional subspace $\cl F$ of $\cl H$ 
such that $\cl M\cap \cl F=\{0\}$, $Th\in \cl M$ implies $h\in \cl M\oplus\cl F$, and $||Th||_{\cl M}\ge ||h||_{\oplus}.$ If $P$ and $L$, 
respectively, are the 
orthogonal projections of $\cl M\oplus \cl F$ onto $\cl M\cap T\cl H$ and $\cl F,$ and $Q=I_{\cl M\oplus \cl F} - P,$ then 
$R:=(T^*T)^{-1}T^*P$ is a well-defined 
contraction on $\cl M\oplus \cl F$. Further, for each $m\ge 0,$ every $h\in \cl M$ can be written as  
$$
h=\sum_{k=0}^{m}T^kQR^kh + T^{m+1}R^{m+1}h + T\sum_{k=1}^{m}T^{k-1}LR^kh 
$$
and 
$$ \|h\|^2_\mathcal{M}\ge \sum_{k=0}^m\|QR^kh\|^2_{\cl M} + \sum_{k=1}^m\|LR^kh\|^2_\mathcal{H}$$\\
\end{lemma}

\begin{proof} For $g\in \mathcal{M} \oplus \mathcal{F},$
\begin{eqnarray*}
TRg&=&T(T^*T)^{-1}T^*P(g)\\
&=&T(T^*T)^{-1}T^*Th_0,  \text{ where } Pg=Th_0 \ \ {\rm for \ some} \ h_0\in\mathcal{H} \ \\
&=&Th_0\\ 
&=& Pg.
\end{eqnarray*}
Thus $TRg\in \cl M$, which implies $Rg\in \cl M\oplus \cl F.$ Also, $||Rg||_{\oplus}\le ||TRg||_{\cl M}=||Pg||_{\cl M}\le ||Pg||_{\oplus}\le ||g||_{\oplus}.$ Therefore $R$ is a 
well-defined contraction on $\cl M\oplus \cl F.$  

Fix any $h \in \mathcal{M}.$ Then  
\begin{equation}\label{defect1}
h=Ph + Qh=TRh + Qh    
\end{equation}
and 
\begin{equation}\label{defect2}
\|h\|^2_\mathcal{M}=\|TRh\|^2_\mathcal{M}+ \|Qh\|^2_\mathcal{M} \ge \|Rh\|^2_{\oplus}+ \|Qh\|^2_\mathcal{M}.     
\end{equation}

Since $Rh \in \mathcal{M} \oplus \mathcal{F}$, therefore we can decompose it as 
$$
Rh=P(Rh)+Q(Rh)+L(Rh)=TR^2h+QRh+LRh
$$
and 
$$||Rh||_{\oplus}^2=||TR^2h||_{\cl M}^2+||QRh||_{\cl M}^2+||LRh||_{\cl H}^2.$$ 

Using these in Equations (\ref{defect1}) and (\ref{defect2}), we obtain  

$$
h=Qh+TQRh+T^2R^2h+TLRh
$$
and 
$$
||h||_{\cl M}^2\ge  \sum_{k=0}^1||QR^kh||_{\cl M}^2+||R^2h||_{\oplus}^2+||LRh||_{\cl H}^2. 
$$

Continuing like this we can show that 
$$
h=\sum_{k=0}^mT^kQR^kh+T^{m+1}R^{m+1}h+\sum_{k=1}^mT^kLR^kh
$$
and 
$$
||h||_{\cl M}^2\ge \sum_{k=0}^m||QR^kh||_{\cl M}^2+||R^{m+1}h||_{\oplus}^2+\sum_{k=1}^m||LR^kh||_{\cl H}^2.
$$
for every $m\ge 0.$
This completes the proof.
\end{proof}

\begin{defn} Let $\cl H$ be an RKHS on a set $X$, $\phi$ be a complex-valued function on $X$ such that $\phi \cl H\subseteq \cl H$ and $M_\phi,$ the operator of multiplication with $\phi$ is bounded below on $\cl H.$ Then a vector subspace $\cl M$ of $\cl H$ is said to be nearly 
invariant under division by $\phi$ with defect $p$ if there exists a $p$-dimensional subspace $\cl F$ (which can assumed to have zero intersection with $\cl M)$ of 
$\cl H$ such that $\phi f\in \cl M$ implies 
$f\in \cl M\oplus\cl F$ (algebraic direct sum). The subspace $\cl F$ (unique upto isomorphism) is said to be the defect space of $\cl M.$   
\end{defn}
 
The following is the main theorem of this section. It an extension of our Theorem \ref{MTRKHS} for the finite defect case. 
 
\begin{thm}\label{RKHS-defect} Let $\mathcal{H}$ be an RKHS consisting of analytic functions on $\bb D.$ Let $\phi$ be an inner function such that $\phi(0)=0, \ \phi \cl H\subseteq \cl H,$ and $||h||\le ||\phi h||$ for every $h\in \cl H$. Assume that if $\phi h\in \cl H$ for an analytic function $h$ on $\bb D,$ then $h\in \cl H.$
Let $\cl M$ be a Hilbert space contractively contained in $\cl H$ which is nearly invariant under division by $\phi$ with defect space $\mathcal{F}$ of dimension $p$  such that $\|\phi h\|_\mathcal{M}\ge \|h\|_{\oplus}$ whenever $\phi h \in \mathcal{M}$. 
\begin{enumerate}
\item If $\mathcal{M}\nsubseteq \phi \mathcal{H},$ then there exists a vector subspace $\mathcal{N}$ of $H^{2}(\mathbb{D},\ell^2(I)\oplus \bb C^p)$ invariant under $T^{*}_{\phi}$ such that $\cl M$ is in one-to-one correspondence with $\cl N$ via the linear map $\cl G:\cl N \to \cl M$ given by 
$$
(\cl Gq)(z)=\sum_{i\in I}g_i(z)f_i(z)+\phi(z)\sum_{i=1}^p e_i(z)t_i(z) \ \ ({\rm pointwise}),
$$
where $q=(f,t)\in \cl N$, and $\{g_i:i\in I\}$ and $\{e_i:i=1,\dots, p\}$ are orthonormal basis of $\cl M\ominus \cl M\cap \phi \cl H$ and $\cl F,$ respectively. Moreover, 
$$
||h||_{\cl M}^2\ge ||(f,t)||_2^2=||f||_2^2+||t||_2^2 
$$
for $h \in \mathcal{M}$, where   
$
h=\cl G(f,t).
$

\item If $\mathcal{M} \subseteq \phi \mathcal{H},$
then there exists a vector subspace $\mathcal{N}$ of $H^{2}(\bb D, \bb C^p)$ invariant under $T^{*}_{\phi}$ such that $\cl M$ is in one-to-one correspondence with $\cl N$ via the linear map $\cl G:\cl N \to \cl M$ given by 
$$
\cl G(t)(z)=\phi(z)\sum_{i=1}^p e_i(z)t_i(z) \ \ ({\rm pointwise}),
$$
where $t\in \cl N$ and $\{e_i:i=1,\dots, p\}$ is an orthonormal basis of $\cl F$. Moreover, 
$$
||h||_{\cl M}\ge ||t||_2 
$$
for $h \in \mathcal{M}$, where   
$
h=\cl G(t).
$

\end{enumerate}
\end{thm}

\begin{proof} Let $h\in \cl M.$ Then, using Lemma \ref{defect-l} for $m\ge 0,$
\begin{equation}\label{MT-def1}
h=\sum_{k=0}^m M_{\phi}^mQR^m h+ M^{m+1}_{\phi}R^{m+1}+ M_{\phi}\sum_{k=1}^m M_{\phi}^{k-1}LR^mh
\end{equation}
and 
\begin{equation}\label{MT-def2}
||h||_{\cl M}^2\ge \sum_{k=0}^\infty ||QR^kh||_{\cl M}^2 + \sum_{k=1}^\infty ||LR^kh||_{\cl H},
\end{equation}
where $Q$ and $L$ are the projections of $\cl M\oplus \cl F$ onto $\cl M\ominus (\cl M\cap \phi \cl H)$ and $\cl F$, respectively. Recall that $\cl M\oplus \cl F$ is the Hilbert space $(\cl M+\cl F, ||\cdot||_{\oplus})$, where $||a + b||_{\oplus}^2=||a||_{\cl M}^2+||b||_{\cl H}^2$ for $a\in \cl M$ and $b\in \cl F$.  

Let $\{g_i:i\in I\}$ and $\{e_i:i=1,\dots, p\}$ be  orthonormal basis of $Ran(Q)$ and $Ran(L),$ respectively. Then 
$$
 QR^{k}h=\sum_{i \in I}c_{ki}g_i \text{ and } LR^kh=\sum_{j=1}^{p}d_{kj}e_j  
 $$
for $\{c_{ki}\}_{i\in I}\in \ell^2(I)$ and $\{d_{kj}\}_{j=1}^p\in \bb C^p.$
Using these represenations in Equation (\ref{MT-def1}), we obtain  
$$
h=\sum_{k=0}^{m}\sum_{i \in I}c_{ki}M^{k}_{\phi }g_{i} + M^{m+1}_{\phi}R^{m+1}h +M_\phi\sum_{k=1}^{m} \sum_{j=1}^{p}d_{kj}M^{k-1}_\phi e_j
$$
and 
$$
 \sum_{i \in I}\sum_{k=0}^{\infty}  |c_{ki}|^{2} + \sum_{j=1}^{p}\sum_{k=1}^{\infty}  |d_{kj}|^{2}\le \|h\|^{2}_{\mathcal{M}}.
$$

Thus for every $i \in  I$ and $j\in \{1,2,\dots,p\}$, $f_i=\sum_{k=0}^{\infty}c_{ki}\phi^{k}$ and  $t_j=\sum_{k=1}^{\infty}d_{kj}\phi^{k-1}$ are well-defined functions in 
$C_{\phi}(H^2(\bb D))$ and $\sum_{i\in I}||f_i||_2^2+\sum_{j=1}^p||t_j||_2^2\le ||h||_{\cl M}^2.$ Then using the arguments simiar to the ones used in the proof of 
Theorem \ref{MTRKHS}, we first show that for each $w\in \bb D,$ 

\begin{equation}\label{MT-def3}
\sum_{i \in I}|(g_i f_i)(w)| \le \|Qk_w\|_\mathcal{M}\|h\|_\mathcal{M}\frac{1}{\sqrt{1-|\phi(w)|^2}}
\end{equation}

and 
\begin{equation}\label{MT-def4}
\sum_{j=1}^{p}|(e_j t_j)(w)|\le \|Lk_w\|_\mathcal{H}\|h\|_\mathcal{M}\frac{1}{\sqrt{1-|\phi(w)|^2}},    
\end{equation}
 and then use them to establish that $\sum_{i \in I}g_{i}f_{i}$ and $\phi\sum_{j=1}^{p}e_{j}t_{j}$ are both analytic on $\mathbb{D}$. Lastly, using 
 Equation (\ref{MT-def1}), we conclude that 
$h-\sum_{i \in I}g_{i}f_{i}-\phi\sum_{j=1}^{p}e_{j}t_{j}$ is an analytic function on $\mathbb{D}$ having zero of every order at 0. Hence,  
\begin{equation}\label{MT-def5}
    h=\sum_{i \in I}g_{i}f_{i} +  \phi\sum_{j=1}^{p}e_{j}t_{j} \ \ \text{ on } \ \ \mathbb{D}.
\end{equation}

Note that each $f_i , t_j\in C_\phi(H^2(\bb D))$. Therefore, we have obtained $f=(f_i)_{i\in I}\in H^2(\bb D, \ell^2(I))$ and 
$t=(t_j)_{j=1}^p\in H^2(\bb D,\bb C^p)$ with $f_i, t_j\in C_\phi(H^2(\bb D))$ such that Equation \ref{MT-def5} holds and  
\begin{equation}\label{MR-def6}
 \|f\|^2_2 + \|t\|^2_2=\sum_{i \in I} \|f_i\|^2_2 + \sum_{j=1} ^{p}\|t_j\|^2_2\\
\le \|h\|^2_{\mathcal{M}}.   
\end{equation}

Define 
\begin{equation*}
\begin{split}
\mathcal{N}=\left \{(f,t) \in H^2(\bb D,\ell^2(I)\oplus \bb C^p) : f=(f_i)_{i\in I}, t= (t_j)_{j=1}^p, f_i, t_j\in C_\phi(H^2(\bb D)) \right.\\
\left. {\rm and} \ \exists \ h \in \cl M \text{ such that} \ h=\sum_{i\in I}g_if_i+\phi \sum_{j=1}^p e_jt_j, \ {\rm and \ for \ each} \ i\in I, \right. \\ 
 \left. 1\le j\le p, \ k\ge 0, \ \ \langle{f_i,\phi^k}\rangle=\langle{QR^kh,g_i}\rangle, \langle{t_j,\phi^{k}}\rangle=\langle{LR^{k+1}h,e_j} \rangle \right\}.
\end{split}
\end{equation*}

Clearly $\mathcal{N}$ is a vector subspace of $H^2(\mathbb{D},\ell^2(I)\oplus\mathbb{C}^p),$ and the map $\cl G:\cl N\to \cl M$ given by 
$$
\cl G(f,t)=\sum_{i\in I}g_if_i+\phi \left(\sum_{j=1}^pe_jt_j\right)
$$
is well-defined one-one, onto, and linear. 

Now we will show that $\mathcal{N}$ is invariant under $T^*_\phi$. Let $(f,t)\in \cl N.$ Then, by definition, there exists a $h\in \cl M$ such that 
$$
h=\sum_{i\in I}g_if_i+\phi\sum_{j=1}^pe_jt_j,
$$
and for each $k\ge 0,$ $\langle{QR^kh,g_i}\rangle=\langle{f_i,\phi^k}\rangle$ and $\langle{LR^{k+1}h,e_j}\rangle=\langle{t_j,\phi^k}\rangle$ for every $i\in I, 1\le j\le p.$
Decompose
\begin{eqnarray*}
h&=&Qh + Ph\\
&=&Qh+M_\phi Rh\\
&=&\sum_{i\in I}c_{0i}g_i+\phi(Rh).    
\end{eqnarray*}
Then 
$$
\phi(Rh)=h-\sum_{i\in I}c_{0i}g_i=\phi\left(\sum_{i\in I}g_i\tilde{f_i}\right)+\phi\left(\sum_{j=1}^pe_jt_j\right),
$$
where $f_i-c_{0j}=\phi \tilde{f_i}$. Then
$$
Rh= \sum_{i\in I}g_i\tilde{f_i} + \sum_{j=1}^pe_jt_j.
$$

Then
$$
\sum_{i\in I}g_i\tilde{f_i} + \sum_{j=1}^pe_jt_j=Rh=L(Rh) +(P+Q)(Rh) = \sum_{j=1}^p d_{0j}e_j + (P+Q)(Rh),
$$
since $Rh\in \cl M\oplus \cl F$ and $L(Rh)=\sum_{j=1}^p d_{0j}e_j.$ Therefore, 
$$
\sum_{i\in I}g_i\tilde{f_i} + \phi \left(\sum_{j=1}^pe_j\tilde{t_j}\right)\in \cl M,
$$
where $t_j-d_{0j}=\phi \tilde{t_j}.$ Let $\tilde{f}=(\tilde{f_i})_{i\in I}$ and $\tilde{t}=(\tilde{t_j})_{j=1}^p$ Then, $(\tilde{f},\tilde{t})\in \cl N;$ hence  
$T_\phi^*(f,t)=(T_{\phi}^*f, T_{\phi}^*t)=(\tilde{f},\tilde{t})\in \cl N.$ This establishes that $\cl N$ is $T_\phi ^*$ invariant; hence completes the proof for the case $\cl M\nsubseteq \phi \cl H.$

Lastly, note that $Q=0$ when $\cl M\subseteq \phi\cl H$. Then, the proof for the case $\cl M\subseteq \phi \cl H$ follows simply by repeating the above arguments with $Q=0.$  
\end{proof}

\end{document}